\documentclass[12pt,reqno]{article}
\UseRawInputEncoding

\usepackage[top=2.5cm,bottom=2.5cm,left=2.5cm,right=2.5cm]{geometry}
\usepackage{graphicx}
\usepackage{latexsym}
\usepackage[top=2.5cm,bottom=2.5cm,left=2.5cm,right=2.5cm]{geometry}
\usepackage{graphicx}
\usepackage{latexsym}
\usepackage{amsmath,amssymb}
\usepackage{amscd}

\usepackage[arrow,matrix]{xy}
\usepackage{graphicx}
\usepackage{xcolor}
\usepackage{times}
\usepackage[]{subfigure}
 \usepackage{cite}
 \usepackage{float}

\usepackage{hyperref}

\usepackage[colorinlistoftodos]{todonotes}

\usepackage{amscd,amsmath,amsthm,amssymb}

\theoremstyle{plain}
\newtheorem{theorem}{Theorem}[section]

\newtheorem{corollary}[theorem]{Corollary}

\theoremstyle{definition}
\newtheorem{remark}[theorem]{Remark}
\newtheorem{definition}[theorem]{Definition}
\newtheorem{example}[theorem]{Example}

\begin{document}

\title{\textbf{Ricci-Yamabe solitons on a Walker 3-manifold}}
\author{\large \textit{Abdou Bousso\(^1\)\thanks{E-mail: abdoukskbousso@gmail.com (\textbf{A.Bousso})}}\\ \small \textit{\(^1\)D\'epartement de Math\'ematiques et Informatique, FST, Universit\'e Cheikh Anta Diop
}, \textit{Dakar}, \textit{S\'en\'egal}
\and
\large \textit{Ameth Ndiaye\(^2\)\thanks{E-mail: ameth1.ndiaye@ucad.edu.sn (\textbf{A.Ndiaye})}}\\ \small \textit{\(^2\)D\'epartement de Math\'ematiques, FASTEF, Universit\'e Cheikh Anta Diop}, \textit{Dakar}, \textit{S\'en\'egal}}
\date{}
\maketitle%
\begin{abstract}
This paper is devoted to the study of Ricci-Yamabe solitons on a particular class of Walker manifolds in dimension 3. We consider a Walker metric where the function $f$ depends on the three coordinates. The novelty of our research lies in the fact that the soliton field is found from the Hodge decomposition of De-Rham with the potential function. We classify all Ricci-Yamabe and gradient Ricci-Yamabe soliton in a given Walker 3-manifold by using this decomposition. Many examples are given in this paper for illustrating our results.
\end{abstract}

\vspace{1cm}
\noindent\textbf{ Keywords}:  Ricci-Yamabe solitons, Walker manifold, Ricci flow, Lie derivative, pseudo-Riemannian geometry.

\vspace{0.2cm}
\noindent\textbf{MSC 2020 : } 53C25, 53C21, 53E20.

\section{Introduction}
\label{sec:intro}
Differential geometry was revolutionized by central concepts such as the Ricci flow, introduced by Richard S. Hamilton \cite{hamilton1982}, which offers a powerful method to deform a Riemannian metric. Hamilton's work led to the notion of the Ricci soliton, which plays a role as a singular point or self-similar solution in the flow.
These structures are crucial for understanding the classification of manifolds and the dynamics of geometric flows. At the same time, the Yamabe problem \cite{yamabe1960} seeks to find a conformal metric with a constant scalar curvature, a problem which, like the study of solitons, is part of the research for canonical metrics on Riemannian manifolds.

The study of Ricci solitons has seen major advances, as shown by the work of H.-D. Cao on Ricci's gradient solitons \cite{cao2010} and the research of Petersen and Wylie on their classification \cite{petersen2009, petersen2010}. The concept has been extended to other flows, such as the Ricci-Bourguignon flow \cite{catino2017} and the $h$-solitons \cite{ghahremani2019}, demonstrating the richness of this approach. Our work focuses on Ricci-Yamabe solitons, a generalization that connects the Ricci flow equations with the scalar curvature, providing a more flexible framework.
 In this context, we are interested in a class of pseudo-Riemannian manifolds called Walker manifolds \cite{pravda2002}. Historically important in general relativity for modeling gravitational waves, these manifolds have a very particular geometric structure. They are characterized by the existence of a zero and parallel vector field \cite{hall2004}. 

 The study of Ricci soliton on a Walker manifold interest many geometers. In \cite{Pirhadi}, the authors characterize the generalized Ricci soliton equation on the three-dimensional Lorentzian Walker manifolds. They prove that every generalized Ricci soliton (with some constants) on a three-dimensional Lorentzian Walker manifold is steady. Moreover, non-trivial solutions for strictly Lorentzian Walker manifolds are derived. Finally, they give some conditions on the defining function $f$ under which a generalized Ricci soliton on a three-dimensional Lorentzian Walker manifold to be gradient. Calvaruso et al. \cite{Calvaruso}, investigate Ricci solitons on Lorentzian three-manifolds $(M,g_f )$ admitting a parallel degenerate line field. For several classes of these manifolds they described in terms of the defining function $f$, they prove the existence of non-trivial Ricci soliton. In \cite{Bousso_Walker4}, the same authors give a classification of Ricci-Yamabe solitons on a specific class of 4-dimensional Walker manifolds.
 
 The aim of this paper is to explicitly characterize Ricci-Yamabe solitons on a 3-dimensional Walker manifold, exploring the links between these geometric structures and building on the foundations of Riemannian geometry \cite{petersen2006} and Einstein manifolds \cite{besse1987}. We are conducting a detailed calculation of geometric quantities to establish the system of equations that governs the existence of these solitons.

\section{Preliminaries and Notations}
\label{sec:preliminaires}
In this section we give some basic definitions and we calculate the geometric properties of the three dimensional Walker manifold. At the end we recall the definition of Ricci-Yamabe soliton.
\subsection{Geometry of Walker 3-manfold}

 A Walker manifold of dimension $n$ is a pseudo-Riemannian manifold $(M, g)$ that admits a  null parallel distribution of rank $k < n$. In dimension 3, such a metric can be expressed in a canonical form in local coordinates. In this study, we focus on the Walker metric of 	\textit{pp-wave} type, whose matrix is expressed as:
\begin{eqnarray}\label{Metric} g =  2 dx^1 \otimes dx^3 + \varepsilon (dx^2)^2 + f(x^1, x^2, x^3) (dx^3)^2 
\end{eqnarray}
with $\varepsilon = \pm 1$ and $f$ is a smooth function of the variables $x^1, x^2, x^3$.\\
The matrix of the metric of a three-dimensional Walker manifold 
$(M, g)$ with coordinates $(x^1, x^2, x^3)$ is expressed as
\begin{eqnarray}\label{Metric1}
g_{ij}=\left(
\begin{array}{ccc}
0 & 0 & 1 \\
0 & \varepsilon & 0  \\
1 & 0 & f
\end{array}\right)\,\,\,\, 
\text{with inverse}\,\,\, \, 
g^{ij}=\left(
\begin{array}{ccc}
-f & 0 & 1 \\
0 & \varepsilon & 0  \\
1 & 0 & 0 
\end{array}\right)
\end{eqnarray}
thus we have 
$D = \mathrm{Span}({\partial_x})$ as the parallel degenerate line field.
Notice that when $\varepsilon=1$ and $\varepsilon=-1$ the Walker 
manifold has signature $(2,1)$ and $(1,2)$ respectively, and 
therefore is Lorentzian in both cases.
Using the formula 

\begin{eqnarray}\label{Chris}\Gamma^i_{jk}=\tfrac12\sum_{\ell=1}^3 g^{i\ell}\big(\partial_j g_{k\ell}+\partial_k g_{j\ell}-\partial_\ell g_{jk}\big)
\end{eqnarray}
we get the non zero components of the Christoffel symbols as

\begin{eqnarray}\label{Chris2}
\Gamma^{1}_{13} =  \frac{1}{2} f_{,1}, \,\, 
\Gamma^{1}_{23} =\frac{1}{2} f_{,2}, \, 
\Gamma^{1}_{33} = \frac{1}{2}(f f_{,1} + f_{,3}), \,
\Gamma^{2}_{33} = -\frac{1}{2\varepsilon} f_{,2}, \,\,
\Gamma^{3}_{33} =  -\frac{1}{2} f_{,1},
\end{eqnarray}
where $f_{, i}=\frac{\partial f}{\partial x^i}, \,\, i=1,2,3$.\\
For computing the Ricci and scalar curvatures, we need the two following equations :
\begin{eqnarray}\label{Ric}
\mathrm{Ric}_{ij}
=\partial_\ell \Gamma^\ell_{ij}-\partial_j \Gamma^\ell_{i\ell}
+\Gamma^\ell_{ij}\Gamma^m_{\ell m}-\Gamma^m_{i\ell}\Gamma^\ell_{jm}.
\end{eqnarray}
and 
\begin{eqnarray}\label{Scalar}
    \mathrm{Scal}=g^{ij}\mathrm{Ric}_{ij}
=\sum_{i=1}^3\sum_{j=1}^3 g^{ij}\mathrm{Ric}_{ij}.
\end{eqnarray}
Using the equations in \eqref{Ric}
and \eqref{Scalar}, we get respectively the components of the Ricci curvature and the scalar curvature of the metric of the Walker 3-manifold as

$$
\mathrm{Ric}_{ij}=
\begin{pmatrix}
0 & 0 & \tfrac12 f_{,11}\\[4pt]
0 & 0 & \tfrac12 f_{,12}\\[4pt]
\tfrac12 f_{,11} & \tfrac12 f_{,12} & \dfrac{\varepsilon f f_{,11}-f_{,22}}{2\varepsilon}
\end{pmatrix}.
$$
and

\begin{eqnarray}\label{Scal1}
\operatorname{Scal} &=&g^{11}R_{11}+g^{22}R_{22}+g^{33}R_{33}
+2\big(g^{12}R_{12}+g^{13}R_{13}+g^{23}R_{23}\big)\nonumber\\[4pt]
&=& (-f)\cdot 0 + \varepsilon\cdot 0 + 0\cdot R_{33}
+2\big(0\cdot0 + 1\cdot\tfrac12 f_{,11} + 0\cdot\tfrac12 f_{,12}\big)\nonumber\\[4pt]
&=& f_{,11}, 
\end{eqnarray}

where $f_{, ij}=\frac{\partial^2f}{\partial x^i\partial x^j}, i=1,2,3$.\\
The Hessian tensor of a smooth function $\mathcal{F}$ on a pseudo-Riemannian manifold is a symmetric tensor of type $(0,2)$ whose components are given by the formula:
\begin{eqnarray}\label{Hessien} (\nabla^2 \mathcal{F})_{ij} = \partial_i \partial_j \mathcal{F} - \Gamma^k_{ij} \partial_k \mathcal{F}. 
\end{eqnarray}
Using \eqref{Chris2} and \eqref{Hessien}, we calculate the components of the Hessian tensor for a function $\mathcal{F}(x^1, x^2, x^3)$ :
\begin{align*}
(\nabla\mathcal{F})_{11} &= \mathcal{F}_{,11}-\Gamma^k_{11}\mathcal{F}_{,k} = \mathcal{F}_{,11}. \\[6pt]
(\nabla\mathcal{F})_{12} &= \mathcal{F}_{,12}-\Gamma^k_{12}\mathcal{F}_{,k} = \mathcal{F}_{,12}. \\[6pt]
(\nabla\mathcal{F})_{13} &= \mathcal{F}_{,13}-\Gamma^k_{13}\mathcal{F}_{,k} = \mathcal{F}_{,13}-\Gamma^1_{13}\mathcal{F}_{,1} = \mathcal{F}_{,13}-\frac{1}{2} f_{,1}\,\mathcal{F}_{,1}. \\[6pt]
(\nabla\mathcal{F})_{22} &= \mathcal{F}_{,22}-\Gamma^k_{22}\mathcal{F}_{,k} = \mathcal{F}_{,22}. \\[6pt]
(\nabla\mathcal{F})_{23} &= \mathcal{F}_{,23}-\Gamma^k_{23}\mathcal{F}_{,k} = \mathcal{F}_{,23}-\Gamma^1_{23}\mathcal{F}_{,1} = \mathcal{F}_{,23}-\frac{1}{2} f_{,2}\,\mathcal{F}_{,1}. \\[6pt]
(\nabla\mathcal{F})_{33} &= \mathcal{F}_{,33}-\Gamma^k_{33}\mathcal{F}_{,k} \\
&= \mathcal{F}_{,33} -\Gamma^1_{33}\mathcal{F}_{,1} -\Gamma^2_{33}\mathcal{F}_{,2} -\Gamma^3_{33}\mathcal{F}_{,3} \\
&= \mathcal{F}_{,33} -\frac{1}{2}\left(f\,f_{,1}+f_{,3}\right)\mathcal{F}_{,1} +\frac{1}{2\varepsilon}f_{,2}\,\mathcal{F}_{,2} +\frac{1}{2} f_{,1}\,\mathcal{F}_{,3}.
\end{align*}
For a vector field $V = V^k \partial_k$, the components of the Lie derivative are given by the formula :
\begin{eqnarray}\label{Lie} (\mathcal{L}_V g)_{ij} = V^k \partial_k g_{ij} + g_{kj} \partial_i V^k + g_{ik} \partial_j V^k.
\end{eqnarray}
By using the equation in \eqref{Lie}, we get
\[
(\mathcal{L}_V g)_{11} = 2\partial_1 V^3.
\]

\[
(\mathcal{L}_V g)_{12} = \varepsilon \partial_1 V^2 + \partial_2 V^3.
\]

\[
(\mathcal{L}_V g)_{13} = \partial_1 V^1 + \partial_3 V^3 + f \partial_1 V^3.
\]

\[
(\mathcal{L}_V g)_{22} = 2\varepsilon \partial_2 V^2.
\]

\[
(\mathcal{L}_V g)_{23} = \partial_2 V^1 + f \partial_2 V^3 + \varepsilon \partial_3 V^2.
\]

\begin{align*}
(\mathcal{L}_V g)_{33} = V^1 f_{,1} + V^2 f_{,2} + V^3 f_{,3} + 2\partial_3 V^1 + 2 f\partial_3 V^3.
\end{align*}
\begin{definition}
For a smooth function $\varphi \in C^\infty(M)$, the Laplace--Beltrami operator is given by
\[
\Delta_g \varphi \;=\; \operatorname{div}_g(\nabla \varphi)
\;=\; g^{ij}\nabla_i\nabla_j \varphi
\;=\; \frac{1}{\sqrt{|\det g|}}\,
\partial_i\!\Big(\sqrt{|\det g|}\, g^{ij}\,\partial_j \varphi\Big).
\]
\end{definition}
With the Walker metric given in \eqref{Metric1}, 
we have $\det g=-\varepsilon$ so $\sqrt{|\det g|}=1$ ( since $\varepsilon=\pm1$). Therefore, we have
\[
\Delta_g \mathcal{F}
= \partial_i\!\big(g^{ij}\,\mathcal{F}_{,j}\big)
= \partial_1\!\big(-f\,\mathcal{F}_{,1}+\mathcal{F}_{,3}\big)
+ \partial_2\!\Big(\varepsilon\,\mathcal{F}_{,2}\Big)
+ \partial_3\!\big(\mathcal{F}_{,1}\big),
\]
from which, in developing, we obtain
\begin{align*}
\Delta_g \mathcal{F}
&= -\,f\,\mathcal{F}_{,11} \;-\; f_{,1}\,\mathcal{F}_{,1}
\;+\; \varepsilon\,\mathcal{F}_{,22}
\;+\; 2\,\mathcal{F}_{,13}.
\end{align*}
\begin{remark}
   We obtain the same expression by taking the trace of the Hessian tensor:
\(
\Delta_g \mathcal{F} = g^{ij}(\nabla^2\mathcal{F})_{ij}.
\)
\end{remark}
\subsection{Ricci-Yamabe soliton}
A Ricci-Yamabe soliton on a pseudo-Riemannian manifold $(M, g)$ is a solution to the following equation:
\begin{eqnarray}\label{YS}
2\beta Ric + \mathcal{L}_V g = (-2\lambda + \mu Scal) g \end{eqnarray}
where $Ric$ is the Ricci tensor, $Scal$  is the scalar curvature, $\mathcal{L}_V g$ is the Lie derivative of the metric $g$ with respect to a vector field $V$, and $\beta, \lambda, \mu$ are real constants.\\ The soliton is said gradient if $V = \nabla \mathcal{F}$ for a function $\mathcal{F} \in C^\infty(M)$. In this case the equation \eqref{YS} becomes
\begin{eqnarray}\label{YS1}
2\beta Ric + 2\nabla^2 \mathcal{F} = (-2\lambda + \mu Scal). g \end{eqnarray}
 The soliton $(M,g,V, \beta, \mu)$ will be called expanding, steady or shrinking if
 $\lambda<0$,$\lambda=0$ or $\lambda>0$, respectively.\\
 On the other hand, given a vector field $V$ over a compact oriented
 Riemannian manifold $(M^n, g)$ the Hodge-de Rham decomposition theorem, \cite{Hodge}, shows that we may decompose $V$ as the sum of a divergence free
 vector field $Y$ and the gradient of a function $\mathcal{F}$, then we set
\begin{eqnarray}\label{Hodge} 
V =Y +\nabla\mathcal{F},
 \end{eqnarray}
 where $\operatorname{div} (Y)=0.$
\section{Main Results}
\label{sec:resultats}
Let $(M, g)$ be a Walker 3-manifold, $\mathcal{F}$ be a smooth function and $Y\in\mathcal{X}(M)$. First we give a classification of Ricci-Yamabe soliton and secondly we give the study of gradient Rici-Yamabe soliton on the Walker 3-manifold $(M, g)$.
\begin{theorem}
\label{thm:main_theorem}
 The manifold 
$(M, g, \nabla \mathcal{F}+Y,\beta,\lambda,\mu)$ with  \(\operatorname{div}(Y)=0\) is Ricci-Yamabe soliton if and only if  $\Delta_g\mathcal{F}=3\left[-\lambda+\left(-\beta+\frac{\mu}{2}\right) \operatorname{Scal}\right]$.
\end{theorem}
\begin{proof}
The proof of the theorem is based on the substitution of the vector field $V = \nabla \mathcal{F}+Y$ where $\operatorname{div}(Y)=0$ in the Ricci-Yamabe system of soliton equations. In other words, according to the Ricci-Yamabe Soliton equation, we have: \[\beta \operatorname{Scal}+\operatorname{tr}(\nabla^2\mathcal{F})=3\left(-\lambda+\frac{\mu}{2}\operatorname{Scal}\right),\]   which gives the  result : $\Delta_g\mathcal{F}=3\left(-\lambda+\left(-\beta+\frac{\mu}{2}\right) \operatorname{Scal}\right)$    .
\end{proof}
\begin{example}
With the metric \eqref{Metric}, we take  \(f(x,y,z)=yze^{-x}\) and the function \(\mathcal{F}(x,y,z)=\mathcal{Y}(y)e^{x+z}\), if \(\left(M,g,\nabla \mathcal{F}+Y,\beta,\lambda,\mu\right)\) is a Ricci-Yamabe soliton, with \(\operatorname{div}(Y)=0\) then we have  $$
\mathcal Y(y)=A\cos\big(\sqrt2\,y\big)+B\sin\big(\sqrt2\,y\big),
$$ or $$
\;\mathcal Y(y)=C e^{\sqrt2\,y}+D e^{-\sqrt2\,y}\;,
$$
with  $A,B,C,D$ some real constants. Indeed, since \(\Delta_g\mathcal{F}=3\left(\lambda-\left(\beta+\frac{\mu}{2}\right) \operatorname{Scal}\right)\) then \[\varepsilon\mathcal{F}_{,22}+2\mathcal{F}_{,13}=3\left(-\lambda+\left(\frac{\mu}{2}-\beta\right)yze^{-x}\right)\Leftrightarrow \varepsilon\mathcal{Y}''(y)e^{x+z}+2\mathcal{Y}(y)e^{x+z}=3\left(-\lambda+\left(\frac{\mu}{2}-\beta\right)yze^{-x}\right).\] \[\varepsilon\mathcal{Y}''(y)+2\mathcal{Y}(y)=3\left(-\lambda+\left(\frac{\mu}{2}-\beta\right)yze^{-x}\right)e^{-x-z}.\] This last one is true if \(\lambda=0\) and \(\mu=2\beta\) so the equation reduces to\[\mathcal{Y}''(y)+2\varepsilon\mathcal{Y}(y)=0.\] 
\begin{itemize}
        \item[i)]
For  $\varepsilon=+1$ :
   the equation becomes $\mathcal Y''+2\mathcal Y=0$. The characteristic polynomial is $r^2+2=0$ so $r=\pm i\sqrt2$. The general real solution is written as
$$
\mathcal Y(y)=A\cos\big(\sqrt2\,y\big)+B\sin\big(\sqrt2\,y\big),
$$
with $A,B\in\mathbb R$.

\item[ii)]  For $\varepsilon=-1$ :
   the equation becomes  $\mathcal Y''-2\mathcal Y=0$.  The characteristic polynomial is $r^2-2=0$  so we have $r=\pm\sqrt2$. The general solution is 
$$
\;\mathcal Y(y)=C e^{\sqrt2\,y}+D e^{-\sqrt2\,y}\;,
$$
with real constants $C, D$.
    \end{itemize}
\end{example}
\begin{theorem}\label{T2}
    For \(\mu\ne 0\), the manifold  \((M,g,V,\beta,\lambda, \mu)\) is Ricci-Yamabe soliton, where \(V(x,y,z)=V^1(x,y,z)\partial x+V^1(x,y,z)\partial y+V^3(x,y,z)\partial z\) if and only if  \[\begin{cases}
    V^1(x,y,z)=\xi(x,z)-\beta\left(\frac{-\varepsilon x^2\partial_{yy}a(y,z)+2x\partial_{y}b(y,z)}{\mu}+ c(y,z)\right)\\
   -\int_{y_0}^y\left(\left(\frac{-\varepsilon x^3}{3\mu}\partial_{ss }a(s,z)+\frac{\partial_sb(s,z)+\lambda}{\mu}x^2+xc(s,z)+v(s,z)\right)\partial_sa(s,z)\right)ds
   -\varepsilon \left(-x\partial_{z}a(y,z)+\int_{y_0}^y\partial_zb(s,z)ds\right)\\ 
   V^2(x,y,z)=\frac{1}{\varepsilon}\left(-x\partial_2a(y,z)+b(y,z)\right)\\
        V^3(x,y,z)=a(y,z)\\
f(x,y,z)=\frac{-\varepsilon x^3}{3\mu}\partial_{yy }a(y,z)+\frac{\partial_yb(y,z)+\lambda}{\mu}x^2+xc(y,z)+v(y,z).        
    \end{cases}\] where $a, b, c, v$ are smooth functions that depend only on the variables $y$ and $z$,
and $\xi$ is a smooth function that depends only on the variables $x$ and $z$.
Under the following constraints:
 \[\begin{cases}
        \beta\left(\frac{\varepsilon ff_{,11}-f_{,22}}{\varepsilon}\right) + (V^1f_{,1}+V^2f_{,2}+V^3f_{,3} + 2\partial_3 V^1 + 2f\partial_3 V^3) -2\partial_y V^2f=0\\
        2\beta(\frac{1}{2} f_{,11}) + \partial_1 V^1 + \partial_3 V^3 -2\partial_yV^2 =0
    \end{cases}.\]
\end{theorem}
\begin{proof} 
Using the components of the Ricci tensor and the Lie derivative, and the equation in  \eqref{YS},  we obtain the following system.
\begin{flalign*}
\begin{cases}
 \partial_1 V^3 = 0 & (1,1) \\
 \varepsilon\partial 1V^2+\partial_2 V^3 = 0 & (1,2) \\
2\beta(\frac{1}{2} f_{,11}) + (\partial_1 V^1 + \partial_3 V^3 + f\partial_1 V^3) = -2\lambda + \mu Scal & (1,3) \\
2\varepsilon\partial_2 V^2 = (-2\lambda + \mu \operatorname{Scal})(\varepsilon) & (2,2) \\
2\beta(\frac{1}{2}f_{,12}) + (\partial_2V^1+ f\partial_2 V^3+\varepsilon\partial_3 V^2 ) = 0 & (2,3) \\
\beta\left(\frac{\varepsilon ff_{,11}-f_{,22}}{\varepsilon}\right) + (V^1f_{,1}+V^2f_{,2}+V^3f_{,3} + 2\partial_3 V^1 + 2f\partial_3 V^3) = (-2\lambda + \mu \operatorname{Scal})f & (3,3)
\end{cases}
\end{flalign*}
where the scalar curvature is given in \eqref{Scal1}.\\
The line 
\((1,1)\) shows that \(V^3(x,y,z)=a(y,z)\). \\
The line
\((1,2)\) implies that  \[V^2(x,y,z)=\frac{1}{\varepsilon}\left(-x\partial_2a(y,z)+b(y,z)\right).\] 
The line \((2,2)\) gives  \(\operatorname{Scal}=f_{,11}=2\frac{\partial yV^2+\lambda}{\mu}\Leftrightarrow f_{,11}=2\frac{-\varepsilon x\partial _{yy}a(y,z)+\partial_yb(y,z)+\lambda}{\mu}
\) \[f(x,y,z)=\frac{-\varepsilon x^3}{3\mu}\partial_{yy }a(y,z)+\frac{\partial_yb(y,z)+\lambda}{\mu}x^2+xc(y,z)+v(y,z).\]
The line \((2,3)\) is equivalent to : \[\partial_y V^1=-\beta\left(f_{xy}\right)-f\partial_yV^3-\varepsilon\partial_zV^2\]
\begin{align*}
   \Leftrightarrow \partial_yV^1&=-\beta\left(\frac{-\varepsilon x^2\partial_{yyy}a(y,z)+2x\partial_{yy}b(y,z)}{\mu}+\partial_y c(y,z)\right)\\
   &-\left(\frac{-\varepsilon x^3}{3\mu}\partial_{yy }a(y,z)+\frac{\partial_yb(y,z)+\lambda}{\mu}x^2+xc(y,z)+v(y,z)\right)\partial_ya(y,z) \\
   &-\varepsilon \left(-x\partial_{yz}a(y,z)+\partial_zb(y,z)\right).
\end{align*}
\begin{align*}
    V^1(x,y,z)&=\xi(x,z)-\beta\left(\frac{-\varepsilon x^2\partial_{yy}a(y,z)+2x\partial_{y}b(y,z)}{\mu}+ c(y,z)\right)\\
    &-\varepsilon \left(-x\partial_{z}a(y,z)+\int_{y_0}^y\partial_zb(s,z)ds\right)\\
   &-\int_{y_0}^y\left(\left(\frac{-\varepsilon x^3}{3\mu}\partial_{ss }a(s,z)+\frac{\partial_sb(s,z)+\lambda}{\mu}x^2+xc(s,z)+v(s,z)\right)\partial_sa(s,z)\right)ds.
\end{align*}
Finally the lines \((1,3)\) and \((3,3)\) are equivalent respectively to  \[2\beta(\frac{1}{2} f_{,11}) + (\partial_1 V^1 + \partial_3 V^3 + f\partial_1 V^3)-2\partial_yV^2 =0\] and \[\beta\left(\frac{\varepsilon ff_{,11}-f_{,22}}{\varepsilon}\right) + (V^1f_{,1}+V^2f_{,2}+V^3f_{,3} + 2\partial_3 V^1 + 2f\partial_3 V^3) -2\partial_y V^2f=0.\]
\end{proof}
\begin{example}\label{E1}
Using the Theorem \ref{T2}, we take the following functions : $a(y,z)=0$, $b(y,z)=y$, $c(y,z)=0$, $v(y,z)=0$, $\xi(x,z)=x$. For the function $f$ we take 
\[
f(x,y,z)=-\frac{\varepsilon x^3}{3\mu}\,\partial_{yy}a(y,z)
+\frac{\partial_yb(y,z)+\lambda}{\mu}x^2+xc(y,z)+v(y,z)
\]
By replacing the above functions, we obtain:
\[
f(x,y,z)=-\frac{\varepsilon x^3}{3\mu}(0)+\frac{\partial_y(y)+\lambda}{\mu}x^2+x(0)+0 = \frac{1+\lambda}{\mu}x^2.
\]
Furthermore, the components of the vector field $V=V^1(x,y,z)\partial_x+ V^2(x,y,z)\partial_y+ V^3(x,y,z)\partial_z$ are :
\[
V^1(x,y,z)=\xi(x,z)-\beta\left(\frac{-\varepsilon x^2 \partial_{yy}a+2x \partial_yb}{\mu}+c\right)
= x-\beta\left(\frac{-\varepsilon x^2(0)+2x(1)}{\mu}+0\right)
\]
\[
= x\left(1-\frac{2\beta}{\mu}\right).
\]
\[
V^2(x,y,z)=\frac{1}{\varepsilon}(-x\partial_ya+b) = \frac{1}{\varepsilon}(-x(0)+y)=\frac{y}{\varepsilon}.
\]
\[
V^3(x,y,z)=a(y,z)=0.
\]
Now let us verifying the constraints. For that we need the following derivatives:
\[
\partial_{xx}f = \partial_{xx}\left(\frac{1+\lambda}{\mu}x^2\right) = \frac{2(1+\lambda)}{\mu}, \quad \partial_{yy}f = 0.
\]
\[
\partial_x V^1 = \partial_x\left(x\left(1-\frac{2\beta}{\mu}\right)\right) = 1-\frac{2\beta}{\mu}, \quad \partial_y V^2 = \partial_y\left(\frac{y}{\varepsilon}\right) = \frac{1}{\varepsilon}.
\]
The other derivatives are zero because the functions do not depend on the corresponding variables.\\
\textbf{First constraint}
\[
\beta\left(\frac{\varepsilon f\partial_{xx}f-\partial_{yy}f}{\varepsilon}\right) + (V^1\partial_xf+V^2\partial_yf+V^3\partial_zf + 2\partial_z V^1 + 2f\partial_z V^3) -2\partial_y V^2 f=0.
\]
By substituting the general expressions:
\begin{align*}
\text{Term 1: } & \beta\left(\frac{\varepsilon \left(\frac{1+\lambda}{\mu}x^2\right)\left(\frac{2(1+\lambda)}{\mu}\right)-0}{\varepsilon}\right) = \frac{2\beta(1+\lambda)^2}{\mu^2}x^2 \\
\text{Term 2: } & \left(x\left(1-\frac{2\beta}{\mu}\right)\left(\frac{2(1+\lambda)}{\mu}x\right)+0+0 + 0 + 0\right) = \frac{2(1+\lambda)}{\mu}\left(1-\frac{2\beta}{\mu}\right)x^2 \\
\text{Term 3: } & -2\left(\frac{1}{\varepsilon}\right)\left(\frac{1+\lambda}{\mu}x^2\right) = -\frac{2(1+\lambda)}{\varepsilon\mu}x^2
\end{align*}
The equation becomes :
\[
\frac{2\beta(1+\lambda)^2}{\mu^2}x^2 + \frac{2(1+\lambda)}{\mu}\left(1-\frac{2\beta}{\mu}\right)x^2 - \frac{2(1+\lambda)}{\varepsilon\mu}x^2 = 0.
\]
What leads to :
\[
2\varepsilon\beta(1+\lambda)^2
+ 2\varepsilon\mu(1+\lambda)\left(1-\frac{2\beta}{\mu}\right)
- 2\mu(1+\lambda) = 0.
\]
\textbf{Second constraint}
\[
2\beta\left(\frac{1}{2}\partial_{xx}f\right)+(\partial_x V^1+\partial_z V^3+f\partial_x V^3) -2\partial_yV^2 =0.
\]
By substituting the general expressions we get :
\[
2\beta\left(\frac{1}{2}\left(\frac{2(1+\lambda)}{\mu}\right)\right) + \left(\left(1-\frac{2\beta}{\mu}\right)+0+0\right) -2\left(\frac{1}{\varepsilon}\right)=0, 
\]
\[
\frac{2\beta(1+\lambda)}{\mu} + 1-\frac{2\beta}{\mu} - \frac{2}{\varepsilon}=0.
\]
We conclude that the Theorem \ref{T2} remains true for the previous choices, provided that the constants meet the following two conditions :
\[
2\varepsilon\beta(1+\lambda)^2\,
+ 2\varepsilon\mu(1+\lambda)\left(1-\frac{2\beta}{\mu}\right)
- 2\mu(1+\lambda) = 0,
\]
and 
\[
\frac{2\beta(1+\lambda)}{\mu} + 1-\frac{2\beta}{\mu} - \frac{2}{\varepsilon}=0.
\]
\end{example}
\begin{remark}
For $\varepsilon=\lambda=\beta=1$ and $\mu=2$, both constraints in Example \ref{E1} are satisfied.
\end{remark}
\begin{theorem}\label{T3}
   The manifold  \((M,g, V,\beta,\lambda,0)\) is a Ricci-Yamabe soliton if and only if the components of the vector field are : \[\begin{cases}
    V^1(x,y,z)=-\beta \partial_{x}f(x,y,z)-yx\partial_z\mathcal{Z}_1(z)-x\partial_z\mathcal{Z}_2(z)-2\lambda x+\xi(y,z)\\
    V^2(x,y,z)=-\varepsilon\mathcal{Z}_1(z)x-\lambda y +\mathcal{Z}_3(z)\\
    V^3(x,y,z)=\mathcal{Z}_1(z) y+\mathcal{Z}_2(z)
\end{cases},\]  where the functions \(\mathcal{Z}_{k\in\{1,2,3,4\}}\) are smooth functions depending only on the variable \(z\), \(\xi\) is smooth function depending only on the variables \(y\), \(z\) and \[\begin{cases}
    -2x\partial_z\mathcal{Z}_1(z)+\partial_y\xi(y,z)+f(x,y,z) \mathcal{Z}_1(z)+\varepsilon\partial_z\mathcal{Z}_3(z)=0\\
    \beta\left(\frac{\varepsilon ff_{,11}-f_{,22}}{\varepsilon}\right) + (V^1f_{,1}+V^2f_{,2}+V^3f_{,3} + 2\partial_3 V^1 + 2f\left(\partial_3 V^3 +\lambda \right)=0.
\end{cases}\]
\end{theorem}
\begin{proof}
    The equation \eqref{YS} gives the following system : \begin{flalign*}
\begin{cases}
 \partial_1 V^3 = 0 &( \ell_1 )\\
 \varepsilon\partial 1V^2+\partial_2 V^3 = 0 & (\ell_2) \\
2\beta(\frac{1}{2} f_{,11}) + (\partial_1 V^1 + \partial_3 V^3 + f\partial_1 V^3) = -2\lambda  & (\ell_3)\\
\partial_2 V^2 = -\lambda  & (\ell_4)\\
2\beta(\frac{1}{2}f_{,12}) + (\partial_2V^1+ f\partial_2 V^3+\varepsilon\partial_3 V^2 ) = 0 & (\ell_5) \\
\beta\left(\frac{\varepsilon ff_{,11}-f_{,22}}{\varepsilon}\right) + (V^1f_{,1}+V^2f_{,2}+V^3f_{,3} + 2\partial_3 V^1 + 2f\partial_3 V^3) = -2\lambda f & (\ell_6)
\end{cases}
\end{flalign*}
The first line \((\ell_1)\) implies \(V^3(x,y,z)=a(y,z)\) and the second \((\ell_2)\) gives \(V^2(x,y,z)=-\varepsilon  x\partial ya(y,z)+b(y,z)\). \\
The line \((\ell_4)\) shows that \[\begin{cases}
    \partial_{yy} a(y,z)=0\\
    \partial_y b(y,z)=-\lambda
\end{cases}\Leftrightarrow\begin{cases}
    a(y,z)=\mathcal{Z}_1(z)y+ \mathcal{Z}_2(z)\\
    b(y,z)=-\lambda y+\mathcal{Z}_3(z)
\end{cases}\] so we have \[\begin{cases}
    V^2(x,y,z)=-\varepsilon\mathcal{Z}_1(z)x-\lambda y +\mathcal{Z}_3(z)\\
    V^3(x,y,z)=\mathcal{Z}_1(z) y+\mathcal{Z}_2(z).
\end{cases}.\]
The line \((\ell_3)\) implies \(\partial_xV^1(x,y,z)=-\beta \partial_{xx}f(x,y,z)-y\partial_z\mathcal{Z}_1(z)-\partial_z\mathcal{Z}_2(z)-2\lambda\) \[\Leftrightarrow V^1(x,y,z)=-\beta \partial_{x}f(x,y,z)-yx\partial_z\mathcal{Z}_1(z)-x\partial_z\mathcal{Z}_2(z)-2\lambda x+\xi(y,z).\]
The line \((\ell_5)\) gives \[-2x\partial_z\mathcal{Z}_1(z)+\partial_y\xi(y,z)+f(x,y,z) \mathcal{Z}_1(z)+\varepsilon\partial_z\mathcal{Z}_3(z)=0.\]
\end{proof}
\begin{remark}
In the Theorem \ref{T3}, if there is a smooth function $\varphi$ depending only on the \(y\) and \(z\) such that \(f:=\varphi\) then the function \(\mathcal{Z}_1\) becomes a constant \(k\) and \[V^1(x,y,z)=-x\mathcal{Z}'_2(z)-2\lambda x -k\int_{y_0}^y\varphi(s,z)ds-\varepsilon y\mathcal{Z}_3'(z)+\mathcal{Z}_4(z).\] 
\end{remark}
\begin{corollary}\label{C1}
   In the Theorem \ref{T3}, if \(\mathcal{Z}_1\) is a zero  function then the components of the vector field \(V\) are : \[\begin{cases}
    V^1(x,y,z)=-\beta \partial_{x}f(x,y,z)-x\mathcal{Z}'_2(z)-2\lambda x-\varepsilon y\mathcal{Z}'_3(z)+\mathcal{Z}_4(z)\\
    V^2(x,y,z)=-\lambda y +\mathcal{Z}_3(z)\\
    V^3(x,y,z)=\mathcal{Z}_2(z)
\end{cases}.\]  and the function  \(f\) is solution of the following equation \[\beta\left(\frac{\varepsilon ff_{,11}-f_{,22}}{\varepsilon}\right) + (V^1f_{,1}+V^2f_{,2}+V^3f_{,3} + 2\partial_3 V^1 + 2f\left(\partial_3 V^3 +\lambda \right)=0.\]
\end{corollary}
\begin{example}
   In the Corollary \ref{C1}, we take   \(f(x,y,z)=\mathcal{
    Z}_5(z)\), \(\lambda=0\) \(\mathcal{Z}_3(z)=az+b\) et \(\mathcal{Z}_4(z)=c\in\mathbb{R}\).
    The constraint becomes 
    
    \[\mathcal{
    Z}_2(z)\mathcal{
    Z}_5'(z)-2x\mathcal{
    Z}''_2(z)+\mathcal{
    Z}_5(z)\mathcal{
    Z}'_2(z)=0\] \[\begin{cases}
        \mathcal{
    Z}_2(z)=a_1z+a_2\\
    \mathcal{
    Z}_5(z)=\varepsilon\left(a_1z+a_2\right)^{-1}.
    \end{cases}\]
    The Corollary \ref{C1} is true if \(\mathcal{
    Z}_5(z)=\varepsilon\left(a_1z+a_2\right)^{-1}\) is well defined.
\end{example}
\begin{corollary}\label{C2}
 If we consider that the function $\mathcal{Z}_1$ is a nonzero smooth function then \[
f(x,y,z) = \frac{2x\,\partial_z \mathcal{Z}_1(z) - \partial_y \xi(y,z) - \varepsilon\,\partial_z \mathcal{Z}_3(z)}{\mathcal{Z}_1(z)}
\] and \[V^1(x,y,z)=-\beta \frac{2\partial_z\mathcal{Z}_1(z)}{\mathcal{Z}_1(z)}-yx\partial_z\mathcal{Z}_1(z)-x\partial_z\mathcal{Z}_2(z)-2\lambda x+\xi(y,z),\]
with the constraint \begin{eqnarray*}\beta\varepsilon \frac{\partial_{yyy}\xi(y,z)}{\mathcal{Z}_1(z)}+V^1(x,y,z)\frac{\mathcal{Z}_1'(z)}{\mathcal{Z}_1(z)}-V^2(x,y,z)\left(\frac{\partial_{yy}\xi(y,z)}{\mathcal{Z}_1(z)}\right)\\+V^3(x,y,z)\partial_z f(x,y,z)+ 2\partial_3 V^1 + 2f\left(\partial_3 V^3 +\lambda \right)=0.\end{eqnarray*}
\end{corollary}
\begin{example}
 For the Corollary \ref{C2},
we fix 
\[
\varepsilon=1,\qquad \lambda=1,
\]
and the functions 
\[
\mathcal{Z}_1(z)=1,\qquad \mathcal{Z}_2(z)=0,\qquad \mathcal{Z}_3(z)=z^2.
\]
We choose an affine function $\xi$ on the variable $y$ :
\[
\xi(y,z)=a(z)\,y+b(z),\qquad a(z)=\mathcal{Z}_3'(z)=2z,\qquad b(z)=2z^2+b_0.
\]
Therefore we have 
\[
\xi(y,z)=2zy+2z^2+b_0.
\]
\textbf{Calculation of  $f$}\\
We have
\[
f(x,y,z)=\frac{2x\,\partial_z \mathcal{Z}_1(z) - \partial_y \xi(y,z) - \varepsilon\,\partial_z \mathcal{Z}_3(z)}{\mathcal{Z}_1(z)}.
\]
But 
\[
\partial_z \mathcal{Z}_1(z)=0,\qquad \partial_y \xi(y,z)=2z,\qquad \partial_z \mathcal{Z}_3(z)=2z,
\]
so we obtain
\[
f(x,y,z)=-2z-2z=-4z.
\]
\textbf{Calculation of the components of the vector field $V$}
\[
V^1(x,y,z)=-x\mathcal{Z}_2'(z)-2\lambda x+\xi(y,z)=-2x+2zy+2z^2+b_0,
\]
\[
V^2(x,y,z)=-\varepsilon \mathcal{Z}_1(z)x-\lambda y+\mathcal{Z}_3(z)=-x-y+z^2,
\]
\[
V^3(x,y,z)=\mathcal{Z}_1(z)y+\mathcal{Z}_2(z)=y.
\]
\textbf{Verification of the Constraint }\\
The constraint to verify is
\[
\beta\varepsilon \frac{\partial_{yyy}\xi}{\mathcal{Z}_1}
+V^1\frac{\mathcal{Z}_1'}{\mathcal{Z}_1}
-V^2\left(\frac{\partial_{yy}\xi}{\mathcal{Z}_1}\right)
+V^3\,\partial_z f
+2\,\partial_z V^1
+2f\big(\partial_z V^3+\lambda\big)=0.
\]
By computation we have 
\[
\partial_{yy}\xi=0,\qquad \partial_{yyy}\xi=0,
\]
\[
\mathcal{Z}_1'(z)=0,
\]
\[
\partial_z f(x,y,z)=\partial_z(-4z)=-4,
\]
\[
\partial_z V^1(x,y,z)=\partial_z(-2x+2zy+2z^2+b_0)=2y+4z,
\]
\[
\partial_z V^3(x,y,z)=\partial_z(y)=0.
\]
So the constraint is 
\[
0+0-0+V^3(-4)+2(2y+4z)+2(-4z)(0+1).
\]
By replacing $V^3=y$ :
\[
-4y+4y+8z-8z=0.
\]
Thus, for the choices
\[
\mathcal{Z}_1(z)=1,\quad \mathcal{Z}_2(z)=0,\quad \mathcal{Z}_3(z)=z^2,\quad 
\xi(y,z)=2zy+2z^2+b_0,\quad \varepsilon=\lambda=1,
\]
the constraint is exactly satisfied.
\end{example}
\begin{theorem}
There is a smooth nonzero vector field $V$ and a function $f$ that corresponds to the third line and third column of the metric (\ref{Metric}), such that $(M, g, V, \beta, \lambda, \mu)$ is a Ricci–Yamabe soliton under certain constant constraints $\beta$, $\lambda$, $\varepsilon$ and $\mu$.
\end{theorem}
 \begin{proof}
    The demonstration is based on the previous examples.
 \end{proof}
 \begin{theorem}\label{t1}
 Let \(\mathcal
 F\) be a smooth function on \(M\) depending only on the variables \(y\) and \(z\). \((M,g,\nabla\mathcal{F}, \beta,\lambda,\mu)\) is a gradient Ricci-Yamabe soliton  for all constants \(\lambda\) and \(\beta\) non zero if and only if \(\beta\ne \mu\). Moreover \[
\begin{aligned}
\mathcal{F}(x,y,z)&=\frac{\varepsilon\lambda\beta}{4(\mu-\beta)}y^2+a\,y+b,\\[6pt]
f(x,y,z)&=\frac{\lambda}{\mu-\beta}x^2+R(z)x+C(z)\int \exp\left(\frac{a}{\beta}y + \frac{\lambda}{4\varepsilon(\mu-\beta)}y^2\right)dy + D(z).
\end{aligned}
\]
  where \(D\), \(C\), \(R\) depending uniquely on the variables \(z\) and \(a\), \(b\) are real constants.
     
 \end{theorem}
 \begin{proof}
By definition \((M,g,\nabla\mathcal{F}, \beta,\lambda,\mu)\) is a gradient Ricci-Yamabe soliton  if and only if \[\beta\operatorname{Ric}+\nabla^2\mathcal{F}=(-\lambda+\frac{\mu}{2}\operatorname{Scal})g\]
which is equivalent to the following system :
\[
\begin{cases}
\partial_{xx}\mathcal{F}=0 & (l_1)\\[4pt]
\partial_{xy}\mathcal{F}=0 & (l_2)\\[4pt]
\partial_{xz}\mathcal{F}-\tfrac12(\partial_x f)(\partial_x \mathcal{F})=-\lambda+(\frac{\mu-\beta}{2})\partial_{xx}f & (l_3)\\[4pt]
\partial_{yy}\mathcal{F}=\varepsilon(-\lambda+\frac{\mu}{2} \partial_{xx}f) & (l_4)\\[4pt]
\partial_{yz}\mathcal{F}-\tfrac12(\partial_y f)(\partial_x \mathcal{F})=-\frac{\beta }{2}\partial_{xy}f & (l_5)\\[4pt]
\partial_{zz}\mathcal{F}-\tfrac12\!\left(f\partial_x f+\partial_z f\right)\partial_x\mathcal{F} +\tfrac{1}{2\varepsilon}(\partial_y f)(\partial_y\mathcal{F})
+\tfrac12 (\partial_x f)(\partial_z\mathcal{F})=\\-\frac{\beta}{2}\left(f\partial_{xx}f-\varepsilon\partial_{yy}f\right)+(-\lambda+\frac{\mu}{2}\partial_{xx}f)f & (l_6).
\end{cases}
\]
The lines $(l_1)$ and $(l_2)$ remain true according to the assumption of \(\mathcal{F}\).\\
We put 
\[
 \mathcal{F}(x,y,z)= B(y,z).
\]
\textbf{ Condition of $(l_3)$}\\
We have
\[
\partial_{xz}\mathcal{F}= \partial_x \mathcal{F}=0.
\]
So we have:
\[
0= -2\lambda+(\mu-\beta)\partial_{xx}f. \tag{$\star$}
\]
the equation $(\star)$ impose that
\[
(\mu-\beta)\partial_{xx}f = 2\lambda,
\] 
which is directly equivalent to $\mu\ne \beta$ because otherwise $\lambda=0$ leads to absurdity.
So we get 
\[
 f(x,y,z)=\frac{\lambda}{\mu-\beta}x^2+u(y,z)x+v(y,z). 
\]
\textbf{ Condition of $(l_4)$}\\
We have
\[
\partial_{yy}\mathcal{F}=B_{yy}(y,z), \qquad
\varepsilon(-\lambda+\frac{\mu }{2}\partial_{xx}f)=\varepsilon\left(-\lambda+\frac{\lambda\mu}{\mu-\beta}\right)=\frac{\varepsilon\lambda\beta}{2(\mu-\beta)}.
\]
Then we get
\[
B_{yy}(y,z)=\frac{\varepsilon\lambda\beta}{2(\mu-\beta)} \quad \Rightarrow \quad 
 B(y,z)=\frac{\varepsilon\lambda\beta}{4(\mu-\beta)}y^2+P(z)\,y+Q(z).
\]
\textbf{ Condition of $(l_5)$}\\
We have
\[
\partial_{yz}\mathcal{F}=P'(z), \qquad \partial_x\mathcal{F}=0,
\]
so we get
\[
P'(z) = -\frac{\beta }{2}\,\partial_{xy}f.
\]
But 
\[
\partial_x f =2 \frac{\lambda}{\mu-\beta}x+u(y,z), \quad \Rightarrow \quad \partial_{xy}f=u_y(y,z).
\]
So we have
\[
P'(z)=-\frac{\beta}{2} u_y(y,z).
\]
It follows that
\[
 u(y,z)=-\tfrac{2P'(z)}{\beta}\,y + R(z).
\]
By compiling the results, we obtain :
\[
\begin{aligned}
\mathcal{F}(x,y,z)&=\frac{\varepsilon\lambda\beta}{4(\mu-\beta)}y^2+P(z)\,y+Q(z),\\[6pt]
f(x,y,z)&=\frac{\lambda}{\mu-\beta}x^2+\Big(-\tfrac{2P'(z)}{\beta}y+R(z)\Big)x+v(y,z),
\end{aligned}
\]
\textbf{ Condition of $(l_6)$}\\
We have the following differential equation, where the functions $\mathcal{F}$ and $f$ are given by:
\begin{eqnarray*}
\partial_{zz}\mathcal{F}-\frac{1}{2}\left(f\partial_x f+\partial_z f\right)\partial_x\mathcal{F} +\frac{1}{2\varepsilon}(\partial_y f)(\partial_y\mathcal{F}) +\frac{1}{2} (\partial_x f)(\partial_z\mathcal{F}) &=& -\frac{\beta}{2}\left(f\partial_{xx}f-\varepsilon\partial_{yy}f\right)\nonumber\\
&&+\left(-\lambda+\frac{\mu}{2}\partial_{xx}f\right)f
\end{eqnarray*}
with
\[
\begin{aligned}
\mathcal{F}(x,y,z)&=\frac{\varepsilon\lambda\beta}{4(\mu-\beta)}y^2+P(z)\,y+Q(z),\\[6pt]
f(x,y,z)&=\frac{\lambda}{\mu-\beta}x^2+\Big(-\tfrac{2P'(z)}{\beta}y+R(z)\Big)x+v(y,z).
\end{aligned}
\]
\textbf{Calculation of the derivatives}\\
The partial derivatives of $\mathcal{F}$ and $f$ are:
\begin{itemize}
    \item $\partial_x \mathcal{F} = 0$
    \item $\partial_y \mathcal{F} = \frac{\varepsilon\lambda\beta}{2(\mu-\beta)}y+P(z)$
    \item $\partial_z \mathcal{F} = P'(z)y+Q'(z)$
    \item $\partial_{zz} \mathcal{F} = P''(z)y+Q''(z)$
    \item $\partial_x f = \frac{2\lambda}{\mu-\beta}x - \frac{2P'(z)}{\beta}y + R(z)$
    \item $\partial_y f = -\frac{2P'(z)}{\beta}x + \partial_y v$
    \item $\partial_{xx} f = \frac{2\lambda}{\mu-\beta}$
    \item $\partial_{yy} f = \partial_{yy} v$
\end{itemize}
\textbf{Simplification of the equation}\\
Since the term $\partial_x \mathcal{F}$ is zero, the second term on the left side of the equation cancels out. In addition, the right side (RHS) is considerably simplified.
\begin{itemize}
    \item The equation becomes :
    \[
    \partial_{zz}\mathcal{F} + \frac{1}{2\varepsilon}(\partial_y f)(\partial_y\mathcal{F}) + \frac12 (\partial_x f)(\partial_z\mathcal{F}) = -\frac{\beta}{2}\left(f\partial_{xx}f-\varepsilon\partial_{yy}f\right)+\left(-\lambda+\frac{\mu}{2}\partial_{xx}f\right)f
    \]
    \item The $(RHS)$ can be rewritten as :
    \[
    RHS = f\left[-\frac{\beta}{2}\partial_{xx}f+\left(-\lambda+\frac{\mu}{2}\partial_{xx}f\right)\right] + \frac{\beta\varepsilon}{2}\partial_{yy}f
    \]
    Using $\partial_{xx}f = \frac{2\lambda}{\mu-\beta}$, the expression in brackets is :
    \[
    -\frac{\beta}{2}\frac{2\lambda}{\mu-\beta} + \left(-\lambda+\frac{\mu}{2}\frac{2\lambda}{\mu-\beta}\right) = -\frac{\beta\lambda}{\mu-\beta} + \frac{-\lambda(\mu-\beta)+\mu\lambda}{\mu-\beta} = -\frac{\beta\lambda}{\mu-\beta} + \frac{\lambda\beta}{\mu-\beta} = 0
    \]
   Therefore, the right side is reduced to:
    \[
    RHS = \frac{\beta\varepsilon}{2}\partial_{yy}f = \frac{\beta\varepsilon}{2}\partial_{yy}v
    \]
\item The expression on the left-hand side of the equation $(LHS)$ is given by:
\[
LHS = \partial_{zz}\mathcal{F} + \frac{1}{2\varepsilon}(\partial_y f)(\partial_y\mathcal{F}) + \frac12 (\partial_x f)(\partial_z\mathcal{F})
\]
By substituting the calculated derivatives, we obtain the complete expression.:
\begin{align*}
    LHS& = (P''(z)y+Q''(z)) + \frac{1}{2\varepsilon}\left(-\frac{2P'(z)}{\beta}x + \partial_y v\right)\left(\frac{\varepsilon\lambda\beta}{2(\mu-\beta)}y+P(z)\right) \\
    &+ \frac12\left(\frac{2\lambda}{\mu-\beta}x - \frac{2P'(z)}{\beta}y + R(z)\right)\left((P'(z)y+Q'(z)\right)
\end{align*}
\textbf{Expanded expression}
To better understand the structure of this expression, we can group the terms according to the monomials of $x$ and $y$.
\begin{itemize}
    \item \textbf{Term on $xy$ :}
    \[
    -\frac{\lambda P'(z)}{2\left(\mu-\beta\right)} + \frac{\lambda P'(z)}{\mu-\beta} = \frac{\lambda P'(z)}{2\left(\mu-\beta\right)}
    \]
This term cancels itself out naturally.
    \item \textbf{Term on $x$ :}
    \[
    -\frac{P'(z)P(z)}{\varepsilon\beta} + \frac{\lambda Q'(z)}{\mu-\beta}
    \]

    \item \textbf{Term on $y^2$ :}
    \[
    -\frac{\left(P'(z)\right)^2}{\beta}
    \]

    \item \textbf{Term on $y$ :}
    \[
    P''(z) + \frac{\lambda\beta}{4(\mu-\beta)}\partial_y v - \frac{P'(z)Q'(z)}{\beta} + \frac12 R(z)P'(z)
    \]

    \item \textbf{Constant term (independent of $x$ and $y$) :}
    \[
    Q''(z) + \frac{P(z)}{2\varepsilon}\partial_y v + \frac12 R(z)Q'(z)
    \]
\end{itemize}
By combining all these terms, the complete expression on the left side is:
\[
\begin{aligned}
LHS &= \left(-\frac{\left(P'(z)\right)^2}{\beta}\right)y^2+ \left(-\frac{P'(z)P(z)}{\varepsilon\beta} + \frac{\lambda Q'(z)}{\mu-\beta}\right)x+\frac{\lambda P'(z)}{2\left(\mu-\beta\right)}xy \\
&+ \left(P''(z) + \frac{\lambda\beta}{4(\mu-\beta)}\partial_y v - \frac{P'(z)Q'(z)}{\beta} + \frac12 R(z)P'(z)\right)y \\
&+ \left(Q''(z) + \frac{P(z)}{2\varepsilon}\partial_y v + \frac12 R(z)Q'(z)\right)
\end{aligned}.
\]
\end{itemize}
So the constraints are:
\[\begin{cases}
\frac{\lambda P'(z)}{2\left(\mu-\beta\right)}=0\\
    \frac{P'(z)P(z)}{\varepsilon\beta} =\frac{\lambda Q'(z)}{\mu-\beta}\\
    \left(-\frac{\left(P'(z)\right)^2}{\beta}\right)y^2 + \left(P''(z) + \frac{\lambda\beta}{4(\mu-\beta)}\partial_y v - \frac{P'(z)Q'(z)}{\beta} + \frac12 R(z)P'(z)\right)y \,+\\
 \left(Q''(z) + \frac{P(z)}{2\varepsilon}\partial_y v + \frac12 R(z)Q'(z)\right)= \frac{\beta\varepsilon}{2}\partial_{yy}v(y,z)
\end{cases}\]
This implies that $P$ is a constant function $a$ since $\lambda \ne 0$, so $Q$ is a constant function $b$ and \[\left(\frac{\varepsilon a}{2}+\frac{\lambda\beta}{4(\mu-\beta)}y\right)\partial_y v=\frac{\beta\varepsilon}{2}\partial_{yy}v(y,z).\]   
Starting from the differential equation, let us set
 $V(y,z) = \partial_y v(y,z)$, so we have :
\[
\left(\frac{\varepsilon a}{2}+\frac{\lambda\beta}{4(\mu-\beta)}y\right)V = \frac{\beta\varepsilon}{2}\partial_{y}V
\]
This equation is a separable differential equation. By rearranging it to separate the variables $V$ and $y$, we obtain:
\[
\frac{dV}{V} = \frac{1}{\frac{\beta\varepsilon}{2}}\left(\frac{\varepsilon a}{2}+\frac{\lambda\beta}{4(\mu-\beta)}y\right)dy
\]
Simplifying the coefficients, the equation becomes:
\[
\frac{dV}{V} = \left(\frac{a}{\beta} + \frac{\lambda}{2\varepsilon(\mu-\beta)}y\right)dy
\]
By integrating both sides, we find the expression for $V(y,z)$ :
\[
\int \frac{dV}{V} = \int \left(\frac{a}{\beta} + \frac{\lambda}{2\varepsilon(\mu-\beta)}y\right)dy
\]
\[
\ln|V| = \frac{a}{\beta}y + \frac{\lambda}{4\varepsilon(\mu-\beta)}y^2 + K(z)
\]
Where $K(z)$ is an integration function depending on $z$. To obtain $V$, we take the exponential:
\[
V(y,z) = \exp\left(\frac{a}{\beta}y + \frac{\lambda}{4\varepsilon(\mu-\beta)}y^2 + K(z)\right)
\]
We can set $C(z) = e^{K(z)}$ to simplify:
\[
V(y,z) = C(z)\exp\left(\frac{a}{\beta}y + \frac{\lambda}{4\varepsilon(\mu-\beta)}y^2\right).
\]
To find $v(y,z)$, we need to integrate $V(y,z)$ with respect to $y$. The resulting integral is a non-elementary integral, meaning it cannot be expressed using standard mathematical functions. The solution is therefore expressed in the form of an integral:
\[
v(y,z) = \int V(y,z)dy + D(z) = C(z)\int \exp\left(\frac{a}{\beta}y + \frac{\lambda}{4\varepsilon(\mu-\beta)}y^2\right)dy + D(z)
\]
where $D(z)$ is a second arbitrary integration function of $z$. This brings us to:
 \[
\begin{aligned}
\mathcal{F}(x,y,z)&=\frac{\varepsilon\lambda\beta}{4(\mu-\beta)}y^2+a\,y+b,\\[6pt]
f(x,y,z)&=\frac{\lambda}{\mu-\beta}x^2+R(z)x+C(z)\int \exp\left(\frac{a}{\beta}y + \frac{\lambda}{4\varepsilon(\mu-\beta)}y^2\right)dy + D(z).
\end{aligned}
\]
\end{proof}
 \begin{theorem}\label{TT}
     Let $\mathcal
F$ is a smooth function on $M$ depending only on the variables $y$ and $z$. If $(M,g,\nabla\mathcal{F}, \beta, \lambda, \mu)$ is a steady gradient Ricci-Yamabe soliton  then we have the following:
 \begin{enumerate}
     \item If \(\beta=\mu\) then we get two cases \begin{enumerate}
         \item if \(\mu=0\) then we have \(\mathcal{F}(x,y,z)=ay+F(z)\)  and \[F''(z)+\frac{\partial x f(x,y,z)}{2}F'(z)+\frac{\varepsilon a}{2}\partial_y f(x,y,z)=0,\] where \(F\) is a smooth function depending only on the variable \(z\);
         \item  if \(\mu\ne 0\) then we have \[\mathcal{F}(x,y,z)=F_1(z)y+F_2(z)\] and \small\begin{eqnarray*}
         f(x,y,z)&=&\frac{-2}{\mu}F_1'(z)x\\
         &&+\Big[A(z)+\int_{y_0}^y\left(\frac{2}{\varepsilon\mu}\left(-F_1''(z)+\frac{(F_1'(z))^2}{\mu}\right)t+\frac{2}{\varepsilon\mu}\left(-F_2''(z)+\frac{F_1'(z)F_2(z)}{\mu}\right)\right)\nonumber\\
         &&\times \exp\left(-\frac{F_1(z)}{\mu}t\right)dt\Big]\exp\left(\frac{F_1(z)}{\mu}y\right).
         \end{eqnarray*}
     \end{enumerate}
     \item  If  \(\beta\ne\mu\) then we get two cases \begin{enumerate}
         \item if \(\beta=0\) then we have \(\mathcal{F}(x,y,z)=bz+c\) and \(f(x,y,z)=xf_1(y,z)+f_2(y,z)\) or \(\mathcal{F}(x,y,z)=ay+F_2(z)\) and  \(f(x,y,z)=xF_3(z)-\frac{2\varepsilon}{a}y\left(F_2''(z)+\frac{a}{2}F_2'(z)F_3(z)\right)+F_4(z)\);
         \item if \(\beta\ne0\) then we have \(\mathcal{F}(x,y,z)=F_2(z),\quad f(x,y,z)=xF_5(z)+\frac{ F_2''(z) + \frac{1}{2}F_5(z)F_2'(z)}{\beta\varepsilon}y^2+F_6(z)y+F_7(z)\),
     \end{enumerate}
 \end{enumerate}
where \(a,\) \(b\), \(c\) are real constants, the functions $F_{k\in\lbrace 1; 2; 3; 4; 5 ; 6; 7\rbrace}$, \(A\) are smooth functions dependent on the variable $z$ and the functions $f_{k\in\lbrace 1; 2\rbrace}$ are smooth functions dependent only on the $y, z$ values.
 \end{theorem}
 \begin{proof}
 \begin{enumerate}
     \item If  \(\beta=\mu\) and using the equation \eqref{YS} then we have the following the system : \[
\begin{cases}
\partial_{yy}\mathcal{F}=\varepsilon\frac{\mu}{2} \partial_{xx}f \\[4pt]
\partial_{yz}\mathcal{F}=-\frac{\mu }{2}\partial_{xy}f \\[4pt]
\partial_{zz}\mathcal{F} +\tfrac{1}{2\varepsilon}(\partial_y f)(\partial_y\mathcal{F})
+\tfrac12 (\partial_x f)(\partial_z\mathcal{F})
=\frac{\mu}{2}\varepsilon\partial_{yy}f.
\end{cases}
\]     \begin{enumerate}
    \item If \(\mu=0\) then the system reduces to  \[
\begin{cases}
\partial_{yy}\mathcal{F}=0 & (S_1)\\[4pt]
\partial_{yz}\mathcal{F}=0 & (S_2)\\[4pt]
\partial_{zz}\mathcal{F} +\tfrac{1}{2\varepsilon}(\partial_y f)(\partial_y\mathcal{F})
+\tfrac12 (\partial_x f)(\partial_z\mathcal{F})
=0 & (S_3).
\end{cases}
\]  The lines  \((S_1)\)  and \((S_2)\) show that \(\mathcal{F}(x,y,z)=ay+F(z)\). So we have \[F''(z)+\frac{\partial x f(x,y,z)}{2}F'(z)+\frac{\varepsilon a}{2}\partial_y f(x,y,z)=0.\]
\item  If  \(\mu\ne0\) then we have   \[
\begin{cases}
\partial_{yy}\mathcal{F}=\varepsilon\frac{\mu}{2} \partial_{xx}f & (s_1)\\[4pt]
\partial_{yz}\mathcal{F}=-\frac{\mu }{2}\partial_{xy}f & (s_2)\\[4pt]
\partial_{zz}\mathcal{F} +\tfrac{1}{2\varepsilon}(\partial_y f)(\partial_y\mathcal{F})
+\tfrac12 (\partial_x f)(\partial_z\mathcal{F})
=\frac{\mu}{2}\varepsilon\partial_{yy}f & (s_3).
\end{cases}
\]  The line  \((s_2)\) shows that the function \(f\) is affine on the variable \(x\) so we have \(f(x,y,z)=f_1(y,z)x+f_2(y,z)\) and the line  \((s_1)\) becomes \(\partial_{yy}\mathcal{F}(x,y,z)=0\) \[\mathcal{F}(x,y,z)=F_1(z)y+F_2(z)\] which brings us back to \[f(x,y,z)=\frac{-2}{\mu}F_1'(z)x+f_2(y,z).\]
The  line \((s_3)\) becomes \[F_1''(z)y+F_2''(z)+\frac{\varepsilon F_1(z)}{2}\partial_yf_2(y,z)-\frac{y}{\mu}\left(F_1'(z)\right)^2-\frac{1}{\mu}F_1'(z)F_2(z)=\frac{\varepsilon\mu}{2}\partial_{yy}f_2(y,z).\]
The differential equation is rewritten in the form:
\begin{eqnarray}\label{de1}
\partial_{yy}f_2 - \frac{F_1(z)}{\mu}\,\partial_y f_2
= \frac{2}{\varepsilon\mu}\!\left(-F_1''(z)+\frac{(F_1'(z))^2}{\mu}\right)y
+\frac{2}{\varepsilon\mu}\!\left(-F_2''(z)+\frac{F_1'(z)F_2(z)}{\mu}\right).
\end{eqnarray}

\medskip
If we set 
\[
r(z) = \frac{F_1(z)}{\mu}, \quad
s_1(z) = \frac{2}{\varepsilon\mu}\!\left(-F_1''(z)+\frac{(F_1'(z))^2}{\mu}\right), \
s_0(z) = \frac{2}{\varepsilon\mu}\!\left(-F_2''(z)+\frac{F_1'(z)F_2(z)}{\mu}\right),
\]
then the equation \eqref{de1} becomes now
\[
\partial_{yy} f_2 - r(z)\,\partial_y f_2 = s_1(z)\,y+s_0(z).
\] The homogeneous solution is: \[f_2^0(y,z)=A(z)\exp(r(z)y)\] so the general solution is : \[f_2(y,z)=\left(A(z)+\int_{y_0}^y\left(s_1(z)t+s_0(z)\right)\exp(-r(z)t)dt\right)\exp(r(z)y).\]

\end{enumerate}
 \item  If  \(\beta\ne\mu\) then \[
\begin{cases}
\partial_{xx}\mathcal{F}=0 & (L_1)\\[4pt]
\partial_{xy}\mathcal{F}=0 & (L_2)\\[4pt]
0=(\frac{\mu-\beta}{2})\partial_{xx}f & (L_3)\\[4pt]
\partial_{yy}\mathcal{F}=\varepsilon\frac{\mu}{2} \partial_{xx}f & (L_4)\\[4pt]
\partial_{yz}\mathcal{F}=-\frac{\beta }{2}\partial_{xy}f & (L_5)\\[4pt]
\partial_{zz}\mathcal{F} +\tfrac{1}{2\varepsilon}(\partial_y f)(\partial_y\mathcal{F})
+\tfrac12 (\partial_x f)(\partial_z\mathcal{F})=-\frac{\beta}{2}\left(f\partial_{xx}f-\varepsilon\partial_{yy}f\right)+\frac{\mu}{2}f\partial_{xx}f & (L_6).
\end{cases}
\] The line\((L_3)\) becomes \(f(x,y,z)=f_1(y,z)x+f_2(y,z)\) and the line \((L_4)\) is now \(\mathcal{F}(x,y,z)=F_1(z)y+F_2(z)\). 
The line \((L_5)\) implies that \(F'_1(z)=-\frac{\beta}{2}\partial_y f_1(y,z)\). 
\begin{enumerate}
    \item  If \(\beta=0\) then we have \(F_1(z)=a\) and \[F''_2(z)+\frac{\varepsilon a}{2}x\partial_yf_1(y,z)+\frac{\varepsilon a}{2}\partial_yf_2(y,z)+\frac{a}{2}f_1(y,z)F_2'(z)=0\]
    \begin{enumerate}
        \item  if \(a=0\) then \(\mathcal{F}(x,y,z)=bz+c\) and \(f(x,y,z)=xf_1(y,z)+f_2(y,z)\)
        \item if \(a\ne 0\) then \(f_1(y,z)=F_3(z)\) and  \(f_2(y,z)=-\frac{2\varepsilon}{a}y\left(F_2''(z)+\frac{a}{2}F_2'(z)F_3(z)\right)+F_4(z)\).
    \end{enumerate}
    \item If \(\beta\ne 0\) then we have \(f_1(y,z)=-\frac{2F_1(z)}{\beta}y+F_5(z)\) and we have the following functions
\begin{itemize}
    \item $f(x,y,z)=f_1(y,z)x+f_2(y,z)$
    \item $\mathcal{F}(x,y,z)=F_1(z)y+F_2(z)$
    \item $f_1(y,z)=-\frac{2F_1(z)}{\beta}y+F_5(z)$
\end{itemize}
The equation to evaluate is:
\begin{eqnarray}\label{ed2} \partial_{zz}\mathcal{F} +\frac{1}{2\varepsilon}(\partial_y f)(\partial_y\mathcal{F}) +\frac12 (\partial_x f)(\partial_z\mathcal{F})=-\frac{\beta}{2}\left(f\partial_{xx}f-\varepsilon\partial_{yy}f\right)+\frac{\mu}{2}f\partial_{xx}f.
\end{eqnarray}
Now we compute the partial derivatives of $\mathcal{F}$ and $f$.\\
\textbf{Partial derivatives of $\mathcal{F}$} :
\begin{itemize}
    \item $\partial_y \mathcal{F} = F_1(z)$
    \item $\partial_z \mathcal{F} = y F_1'(z) + F_2'(z)$
    \item $\partial_{zz} \mathcal{F} = y F_1''(z) + F_2''(z)$
\end{itemize}
\textbf{Partial derivatives of  $f$} :
\begin{itemize}
    \item $\partial_x f = f_1(y,z) = -\frac{2F_1(z)}{\beta}y+F_5(z)$
    \item $\partial_{xx} f = 0$
    \item $\partial_y f = \partial_y(f_1(y,z)x) + \partial_y(f_2(y,z)) = \left(-\frac{2F_1(z)}{\beta}\right)x + \partial_y f_2(y,z)$
    \item $\partial_{yy} f = \partial_y\left(-\frac{2F_1(z)}{\beta}x + \partial_y f_2(y,z)\right) = \partial_{yy}f_2(y,z)$
\end{itemize}
Now we substitute these partial derivatives in the equation \eqref{ed2} and noting that the terms $partial_{xx}f$ are zero, we get 
\begin{eqnarray}\label{eq12}\left(y F_1''(z) + F_2''(z)\right) + \frac{1}{2\varepsilon}\left[\left(-\frac{2F_1(z)}{\beta}\right)x
+ \partial_y f_2(y,z)\right]\left(F_1(z)\right) \\+ \frac{1}{2}\left(-\frac{2F_1(z)}{\beta}y+F_5(z)\right)\left(y F_1'(z) + F_2'(z)\right)\nonumber = -\frac{\beta}{2}\left(0-\varepsilon\partial_{yy}f\right)+\frac{\mu}{2}f(0) 
\end{eqnarray}
By simplifying the right side of the equation \eqref{eq12}, we obtain:
\begin{eqnarray*} \left(y F_1''(z) + F_2''(z)\right) + \frac{1}{2\varepsilon}\left[\left(-\frac{2F_1(z)}{\beta}\right)x + \partial_y f_2(y,z)\right]\left(F_1(z)\right)\\ + \frac{1}{2}\left(-\frac{2F_1(z)}{\beta}y+F_5(z)\right)\left(y F_1'(z) + F_2'(z)\right) 
= \frac{\beta\varepsilon}{2}\partial_{yy}f_2(y,z) 
\end{eqnarray*}
this amounts to saying that \(F_1(z)=0\) and $$ F_2''(z) + \frac{1}{2}F_5(z)F_2'(z) = \frac{\beta\varepsilon}{2}\partial_{yy}f_2(y,z) $$
\(\mathcal{F}(x,y,z)=F_2(z),\qquad f(x,y,z)=xF_5(z)+\frac{ F_2''(z) + \frac{1}{2}F_5(z)F_2'(z)}{\beta\varepsilon}y^2+F_6(z)y+F_7(z)\).

\end{enumerate}
 \end{enumerate}

 \end{proof}
 \begin{example}
For Theorem \ref{TT}, we take the case \(\beta=\mu=0\) and for \(f(x,y,z)=a_1x+H(z)y\) where \(a_1\) is real constant and \(H\) is a smooth function of \(z\). We obtain the following differential equation:
\[F''(z)+\frac{a_1}{2}F'(z)=-\frac{\varepsilon a}{2}H(z).\]

This is a first-order linear ordinary differential equation (ODE) by setting $\mathcal{U}(z) = F'(z)$. Let's perform this substitution to solve it.

 \textbf{The Transformed Equation}

By setting $\mathcal{U}(z) = F'(z)$, we also have $\mathcal{U}'(z) = F''(z)$. Substituting these terms into the original equation gives us a first-order ODE for $\mathcal{U}(z)$:

\[\mathcal{U}'(z) + \frac{a_1}{2}\mathcal{U}(z) = -\frac{\varepsilon a}{2}H(z).\]

This is a standard first-order linear ODE of the form $\mathcal{U}'(z) + P(z)\mathcal{U}(z) = Q(z)$, with $P(z) = \frac{a_1}{2}$ and $Q(z) = -\frac{\varepsilon a}{2}H(z)$.

\textbf{ Finding the Solution for} $\mathcal{U}(z)$

We will use an integrating factor, given by $I(z) = e^{\int P(z)dz}$.
\begin{enumerate}
 \item \textbf{Calculating the Integrating Factor}:
 \[I(z) = e^{\int \frac{a_1}{2}dz} = e^{\frac{a_1}{2}z}.\]

\item \textbf{Multiplying the ODE by the Integrating Factor}:
 \[e^{\frac{a_1}{2}z}\mathcal{U}'(z) + e^{\frac{a_1}{2}z}\frac{a_1}{2}\mathcal{U}(z) = -e^{\frac{a_1}{2}z}\frac{\varepsilon a}{2}H(z).\]
 The left side of the equation is now the derivative of a product:
 \[\frac{d}{dz}\left(e^{\frac{a_1}{2}z}\mathcal{U}(z)\right) = -e^{\frac{a_1}{2}z}\frac{\varepsilon a}{2}H(z).\]

\item \textbf{Integrating both sides with respect to } $z$:
 \[\int \frac{d}{dz}\left(e^{\frac{a_1}{2}z}\mathcal{U}(z)\right)dz = \int -e^{\frac{a_1}{2}z}\frac{\varepsilon a}{2}H(z)dz.\]
 \[e^{\frac{a_1}{2}z}\mathcal{U}(z) = -\frac{\varepsilon a}{2}\int e^{\frac{a_1}{2}z}H(z)dz + C_1,\]
 where $C_1$ is the constant of integration.

\item \textbf{Solving for} $\mathcal{U}(z)$:
 \[\mathcal{U}(z) = e^{-\frac{a_1}{2}z}\left(-\frac{\varepsilon a}{2}\int e^{\frac{a_1}{2}z}H(z)dz + C_1\right).\]

\end{enumerate}

 \textbf{Finding the Solution for}  $F(z)$

Since $\mathcal{U}(z) = F'(z)$, we integrate $\mathcal{U}(z)$ to find $F(z)$:

\[F(z) = \int \mathcal{U}(z)dz = \int \left(e^{-\frac{a_1}{2}z}\left(-\frac{\varepsilon a}{2}\int e^{\frac{a_1}{2}z}H(z)dz + C_1\right)\right)dz + C_2,\]
where $C_2$ is a new constant of integration.
Thus \[\mathcal{F}(x,y,z)=ay+ \int \left(e^{-\frac{a_1}{2}z}\left(-\frac{\varepsilon a}{2}\int e^{\frac{a_1}{2}z}H(z)dz + C_1\right)\right)dz + C_2.\]
\end{example}
 \begin{theorem}
Let \(\mathcal{F}\) be a smooth function on \(M\) depending on the three variables.
     \begin{enumerate}
         \item The manifold \((M,g,\nabla\mathcal{F},0,\lambda,0)\) is a gradient Ricci-Yamabe soliton if and only if   \[\mathcal{F}(x,y,z)=F(z)x-\frac{\varepsilon\lambda}{2}y^2+a(z)y+b(z),\] \[f(x,y,z)=\frac{2y\partial_{z}a(z)+2\partial_{z}b(z)}{F(z)}+\frac{2\left(\lambda+F'(z)\right)}{F(z)}x+c(z)\]
         and \[\partial_{zz} \mathcal{F}-\frac{1}{2}\left(f\,\partial_xf+\partial_zf\right)\partial_x\mathcal{F} +\frac{1}{2\varepsilon}\partial_yf\partial_y\mathcal{F} +\frac{1}{2} \partial_xf\partial_z\mathcal{F}=-\lambda f, \]
         \item for a nonzero constant \(\mu\), the manifold  \((M,g,\nabla\mathcal{F},0,\lambda,\mu)\) is a gradient Ricci-Yamabe soliton if and only if\(\mathcal{F}(x,y,z)=F(z)x+F_2(y,z)\),  \[f(x,y,z)=\frac{2\partial_{z}F_2(y,z)}{F(z)}+\frac{\left(\varepsilon\lambda+\partial_{yy}F_2(y,z)\right)}{\varepsilon\mu}x^2+F_4(z)x+F_5(z)\] and \[\begin{cases}
             \partial_{xz} \mathcal{F}-\frac{1}{2}\partial_xf \partial_{x} \mathcal{F}=-\lambda+\frac{\mu}{2}\partial_{xx}f\\
             \partial_{zz} \mathcal{F}-\frac{1}{2}\left(f\,\partial_xf+\partial_zf\right)\partial_x\mathcal{F} +\frac{1}{2\varepsilon}\partial_yf\partial_y\mathcal{F} +\frac{1}{2} \partial_xf\partial_z\mathcal{F}=(-\lambda+\frac{\mu}{2}\partial_{xx}f)f
         \end{cases},\]
     \end{enumerate}
     where \(F\), \(a\), \(b\), \(c\), \(F_4\), \(F_5\) functions depending on the variable $z$ such that $F$ is nonzero and $F_2$ is a function depending on the variables $y, z$.
 \end{theorem}
 \begin{proof}
     By hypothesis we have the following system :  \[\begin{cases}
         \partial_{xx}\mathcal{F}=0& (c_1)\\
         \partial_{xy} \mathcal{F}=0& (c_2)\\
          \partial_{xz} \mathcal{F}-\frac{1}{2}\partial_xf \partial_{x} \mathcal{F}=-\lambda+\frac{\mu}{2}\partial_{xx}f& (c_3)\\
           \partial_{yy} \mathcal{F}=\varepsilon(-\lambda+\frac{\mu}{2} \partial_{xx}f)& (c_4)\\
           \partial_{yz}\mathcal{F}-\frac{1}{2}\partial_yf\partial_x\mathcal{F}=0& (c_5)\\
            \partial_{zz} \mathcal{F}-\frac{1}{2}\left(f\,\partial_xf+\partial_zf\right)\partial_x\mathcal{F} +\frac{1}{2\varepsilon}\partial_yf\partial_y\mathcal{F} +\frac{1}{2} \partial_xf\partial_z\mathcal{F}=(-\lambda+\frac{\mu}{2}\partial_{xx}f)f& (c_6)
     \end{cases}\]
   The line  \((c_1)\) shows that  \(\mathcal{F}(x,y,z)=F_1(y,z)x+F_2(y,z)\) and the line  \((c_2)\) ensures that \(\mathcal{F}(x,y,z)=F(z)x+F_2(y,z)\) 
     so the line \((c_5)\) becomes \[2\partial_{yz}F_2(y,z)=F(z)\partial_yf(x,y,z)\Leftrightarrow f(x,y,z)=\frac{2\partial_{z}F_2(y,z)}{F(z)}+F_3(x,z).\]
     \begin{enumerate}
         \item For  \(\mu=0\), the line  \((c_4)\) implies that \(F_2(y,z)=\frac{-\varepsilon\lambda}{2}y^2+a(z)y+b(z)\) so we have  \[f(x,y,z)=\frac{2y\partial_{z}a(z)+2\partial_{z}b(z)}{F(z)}+F_3(x,z).\]
        The line  \((c_3)\) becomes \[F'(z)-\frac{F(z)}{2}\partial_xF_3(x,z)=-\lambda\Leftrightarrow F_3(x,z)=\frac{2\left(\lambda+F'(z)\right)}{F(z)}x+c(z)\]
         \(\mathcal{F}(x,y,z)=F(z)x-\frac{\varepsilon\lambda}{2}y^2+a(z)y+b(z)\), \(f(x,y,z)=\frac{2y\partial_{z}a(z)+2\partial_{z}b(z)}{F(z)}+\frac{2\left(\lambda+F'(z)\right)}{F(z)}x+c(z)\)
         \[\partial_{zz} \mathcal{F}-\frac{1}{2}\left(f\,\partial_xf+\partial_zf\right)\partial_x\mathcal{F} +\frac{1}{2\varepsilon}\partial_yf\partial_y\mathcal{F} +\frac{1}{2} \partial_xf\partial_z\mathcal{F}=-\lambda f.\]
         \item For \(\mu\ne 0\),  the line \((c_4)\) implies that \[\partial_{yy}F_2(y,z)=-\varepsilon\lambda+\frac{\varepsilon\mu}{2}\partial_{xx}F_3(x,z)\Leftrightarrow F_3(x,y)=\frac{\left(\varepsilon\lambda+\partial_{yy}F_2(y,z)\right)}{\varepsilon\mu}x^2+F_4(z)x+F_5(z)\]
         \[f(x,y,z)=\frac{2\partial_{z}F_2(y,z)}{F(z)}+\frac{\left(\varepsilon\lambda+\partial_{yy}F_2(y,z)\right)}{\varepsilon\mu}x^2+F_4(z)x+F_5(z)\] \(\mathcal{F}(x,y,z)=F(z)x+F_2(y,z)\).
     \end{enumerate}
 \end{proof}
 \begin{theorem}\label{fin}
   Let $\mathcal{F}$ be a smooth function of $M$ depending on the three variables. For any nonzero $\beta$ constant, \((M,g, \nabla \mathcal{F}, \beta,\lambda, \mu)\) is a gradient Ricci-Yamabe soliton if and only if   \(\mathcal{F}(x,y,z)=F_1(y,z)x+F_2(y,z)\)  and  \[f(x,y,z)=\exp\left(\frac{xF(z)}{\beta}\right)\int_{y_0}^ya(s,z)ds+\frac{2\partial_zF_2(y,z)}{F(z)}+b(x,z)\] with the constraints: \[\begin{cases}
 \partial_{xz} \mathcal{F}-\frac{1}{2}\partial_xf \partial_{x} \mathcal{F}=-\lambda+(\frac{\mu-\beta}{2})\partial_{xx}f \\
           \partial_{yy} \mathcal{F}=\varepsilon(-\lambda+\frac{\mu}{2} \partial_{xx}f)\\
            \partial_{zz} \mathcal{F}-\frac{1}{2}\left(f\,\partial_xf+\partial_zf\right)\partial_x\mathcal{F} +\frac{1}{2\varepsilon}\partial_yf\partial_y\mathcal{F} +\frac{1}{2} \partial_xf\partial_z\mathcal{F}=-\beta\left(f\partial_{xx}f-\varepsilon\partial_{yy}f\right)+(-\lambda+\frac{\mu}{2}\partial_{xx}f)f
     \end{cases}\] where $F$ is a smooth function depending on $z$ that is non-zero, $a$ and $F_2$ are smooth functions depending on the variables $y$, $z$, and $b$ is a smooth function depending on the variables $x$, $z$.
 \end{theorem}
 \begin{proof}
 By the hypothesis we have 
     \[\begin{cases}
         \partial_{xx}\mathcal{F}=0& (m_1)\\
         \partial_{xy} \mathcal{F}=0& (m_2)\\
          \partial_{xz} \mathcal{F}-\frac{1}{2}\partial_xf \partial_{x} \mathcal{F}=-\lambda+(\frac{\mu-\beta}{2})\partial_{xx}f& (m_3)\\
           \partial_{yy} \mathcal{F}=\varepsilon(-\lambda+\frac{\mu}{2} \partial_{xx}f)& (m_4)\\
           \partial_{yz}\mathcal{F}-\frac{1}{2}\partial_yf\partial_x\mathcal{F}=-\frac{\beta }{2}\partial_{xy}f& (m_5)\\
            \partial_{zz} \mathcal{F}-\frac{1}{2}\left(f\,\partial_xf+\partial_zf\right)\partial_x\mathcal{F} +\frac{1}{2\varepsilon}\partial_yf\partial_y\mathcal{F} +\frac{1}{2} \partial_xf\partial_z\mathcal{F}=\\-\beta\left(f\partial_{xx}f-\varepsilon\partial_{yy}f\right)+(-\lambda+\frac{\mu}{2}\partial_{xx}f)f& (m_6)
     \end{cases}\]
  The line  \((m_1)\) shows that \(\mathcal{F}(x,y,z)=F_1(y,z)x+F_2(y,z)\) and the line \((m_2)\) implies that  \(\mathcal{F}(x,y,z)=F(z)x+F_2(y,z)\) 
     so the line \((m_3)\) becomes \[\partial_{yz}F_2(y,z)=-\frac{\beta}{2}\partial_x(\partial_yf(x,y,z))+\frac{F(z)}{2}\partial_yf(x,y,z)\] the homogeneous solution is  \(A_0(x,y,z)=a(y,z)\exp\left(\frac{xF(z)}{\beta}\right)\) and we get \[f(x,y,z)=\exp\left(\frac{xF(z)}{\beta}\right)\int_{y_0}^ya(s,z)ds+\frac{2\partial_zF_2(y,z)}{F(z)}+b(x,z).\]
 \end{proof}
 \begin{example}
  In the Theorem \ref{fin},
we choose 
\[\lambda=1,\qquad
F(z)=1,\qquad a(y,z)=0,\qquad b(x,z)=0,
\]
\[
F_1(y,z)=\alpha y - z, \qquad 
F_2(y,z)=-\tfrac{\varepsilon}{2}y^2 - \tfrac{C}{2}e^{-2z},
\]
where \(\alpha, C\) are constants.

Then we have
\[
\mathcal{F}(x,y,z)=(\alpha y - z)\,x - \tfrac{\varepsilon}{2}y^2 - \tfrac{C}{2}e^{-2z},
\qquad
f(x,y,z)=2\partial_z F_2(y,z)=2C e^{-2z}.
\]

\[
\textbf{Verification of the constraints :}
\]

1. first equation :
\[
\partial_{xz}\mathcal{F}=-\lambda.
\]
Indeed we have \(\partial_x\mathcal{F}=F_1(y,z)=\alpha y - z\), so
\[
\partial_{xz}\mathcal{F}=\partial_z(\alpha y - z)=-1=-\lambda.
\]

2. Second equation :
\[
\partial_{yy}\mathcal{F}=\varepsilon(-\lambda).
\]
We calculate
\[
\partial_{yy}\mathcal{F}=\partial_{yy}\!\left(-\tfrac{\varepsilon}{2}y^2\right)=-\varepsilon,
\]
and since \(\lambda=1\), we obtain 
\[
\partial_{yy}\mathcal{F}=-\varepsilon=\varepsilon(-1)=\varepsilon(-\lambda).
\]

3. Third equation :
\[
\partial_{zz}\mathcal{F}=-\lambda f.
\]
We have
\[
\partial_{zz}\mathcal{F}=\partial_{zz}\!\left(-\tfrac{C}{2}e^{-2z}\right)
=(-\tfrac{C}{2})(-2)^2 e^{-2z}=-2C e^{-2z}.
\]
On the other hand, we have
\[
-\lambda f=-1\cdot (2C e^{-2z})=-2C e^{-2z}.
\]
The two sides coincide.
 \end{example}

\end{document}